\documentclass[a4paper,reqno]{amsart}
\usepackage{graphicx} 

\title{Lower semicontinuity  of nonlocal $L^\infty$ energies on $SBV_0(I)$}

\author[Jos\'e Matias]{Jos\'e Matias}
\address{Departamento de Matem\'atica, Instituto Superior T\'ecnico. University of Lisbon,
Portugal}
\email{jose.c.matias@tecnico.ulisboa.pt}
\author[Pedro Miguel Santos]{Pedro Santos}
\address{Departamento de Matem\'atica, Instituto Superior T\'ecnico. University of Lisbon,
Portugal}
\email{pedro.m.santos@tecnico.ulisboa.pt}

\author[Elvira Zappale]{Elvira Zappale}
\address{Department of Basic and Applied Sciences for Engineering,
Sapienza-University of Rome, via
A. Scarpa,  16 (00161) Roma, Italy and
CIMA, Universidade de \'Evora, Portugal}
\email{elvira.zappale1@uniroma1.it}

\usepackage[T1]{fontenc}
\usepackage[latin1]{inputenc}
\usepackage{amssymb}
\usepackage{amsfonts}
\usepackage{amsmath,mathtools}
\usepackage{mathrsfs}
\usepackage{graphicx}
\usepackage[dvipsnames,usenames]{color}
\usepackage{hyperref}
\usepackage{esint}
\usepackage{verbatim}
\usepackage{bbm}
\usepackage{tikz}
\usepackage{soul}
\definecolor{PineGreen}{rgb}{0.0, 0.26, 0.15}

\newcommand{\R}[1]{\mathbb{R}^{#1}}

\DeclareMathOperator*{\supess}{ess\,sup}

\renewcommand{\leq}{\leqslant}

\newcommand{\average}{{\mathchoice {\kern1ex\vcenter{\hrule
height.4pt width 8pt depth0pt}
\kern-11pt} {\kern1ex\vcenter{\hrule height.4pt width 4.3pt
depth0pt} \kern-7pt} {} {} }}

\mathchardef\emptyset="001F

\providecommand{\U}[1]{\protect\rule{.1in}{.1in}}
\numberwithin{equation}{section}
\newtheorem{definition}{Definition}[section]
\newtheorem{theorem}[definition]{Theorem}

\newtheorem{proposition}[definition]{Proposition}
\newtheorem{remark}[definition]{Remark}

\makeindex

\begin{document}

 \begin{abstract}  
 \vspace{-12pt}   

We characterize the lower-semicontinuity
of nonlocal one-dimensional energies of the type
\[{\rm ess}\!\!\!\!\!\!\!\!\sup_{(s,t) \in I\times I} h([u](s), [u](t)),\] where $I$ is an open and bounded interval in the real line, $u \in SBV_0(I)$ and $[u](r):= u(r^+)- u(r^-)$, with $r\in I$.                   
\vspace{8pt}

 \noindent\textsc{MSC (2020):}  49J45, 26B25, 26A45, 26A51.
 
 \noindent\textsc{Keywords:} Cartesian submaximality, nonlocal $L^\infty$ energies, interfacial energy, lower semicontinuity, nonlocal gradients, supremal functionals.

 \noindent\textsc{Date:} \today.
 \end{abstract}

\maketitle

 \section{Introduction}

 In recent years a great attention has been devoted to nonlocal functionals in the form of double integrals because of their many different applications; see, for instance \cite{BMCP, BP, BDM, BN}, and the bibliography contained therein, among a wide literature. New mathematical topics, inspired by machine
learning issues, presented models  involving  supremal
(or $L^\infty$-)functionals in the nonlocal setting, see \cite{GZ, KS, KRZ, KZ, RB, S} where energies are either of the form 
\[{\rm ess}\!\!\!\!\!\!\!\!\sup_{(x,y) \in \Omega\times \Omega} f(u(x), u(y)),\]
with $\Omega \subset \mathbb R^N $, $f:\mathbb R^m \times \mathbb R^m \to \mathbb R$ a suitable lower semicontinuous and coercive function, and $u \in L^\infty(\Omega;\mathbb R^m)$, possibly satisfying a PDE's constraint, or
\[
{\rm ess}\sup_{x \in \Omega} f(x, u(x), \nabla^\alpha u(x)),\]
with $f:\Omega \times \mathbb R^m\times \mathbb R^{N\times m}\to \mathbb R $ a suitable integrand and $u \in \mathcal S^{\alpha,\infty}_{u_0}(\Omega;\mathbb R^m)$, with $\alpha \in (0,1)$, $\nabla^\alpha u$ being the Riesz weak fractional gradient and  the space $S^{\alpha,\infty}_{u_0}(\Omega;\mathbb R^m) = \{u \in L^\infty (\mathbb R^m), \nabla^\alpha u \in L^\infty(\mathbb R^{m\times N}), u= u_0 \in \mathbb R^N\setminus\Omega\}$.   

\smallskip

The interest towards supremal functionals started with the pioneering works by Aronsson, see \cite{A1,A2,A3,A4} and has been later developed by many authors see \cite{ABP, BJW}, both in connection with the so-called power-law approximation (see, for instance, \cite{CDP} and \cite{BEZ} for developments in wider contexts than classical Sobolev spaces) and in relation with the $\infty$-laplacian type equations. We refer to \cite{AK, KM1, KM2} for recent advances on this topic.

$L^\infty$ energies defined at jump points of piecewise continuous functions on the real line were first analyzed in \cite{ABC}.
Indeed, the $L^1$-lower semicontinuity of the functional
\begin{equation}\label{Gdef}
{\rm ess}\sup_{t \in I} g([u](t)),
\end{equation}
where $I$ is a bounded open interval of the real line, $u$ is a piecewise constant function in the sense of \cite{AFP} and 
\begin{equation}\label{applim}[u](r):= \lim_{\varrho \to r^+} u(\varrho)- \lim_{\varrho \to r^-}u(\varrho),
\end{equation} was characterized in terms of the lower semicontinuity and the {\it submaximality} of the density $g$, which can be seen as the supremal counterpart of $BV$-ellpiticty introduced by \cite{AB} to characterize the lower semicontinuity of elliptic integrands defined on piecewise constant functions.

More specifically, we recall that a function $g : \mathbb R\setminus \{0\} \to \mathbb R$ is said to be submaximal if the following inequality
holds true for all $x_1, x_2 \in \mathbb R \setminus \{0\}$ such that $x_1 \not= -x_2$:
\begin{equation*}
g(x_1 + x_2)\leq \max\{g(x_1), g(x_2)\}.
\end{equation*}
It is also possible to consider $g$ as defined in whole $\mathbb R$, even
though its value in $0$ is never taken into account, by setting $g(0) = \inf g$. Note
that, in \cite{ABC}, $g$ is submaximal and lower semicontinuous in $\mathbb R \setminus \{0\}$ if and only if its
extension to $0$ is submaximal and lower semicontinuous.

Taking into account the models described before, we consider here the following nonlocal energy
\[H(u):={\rm ess}\!\!\!\!\!\!\!\!\!\!\!\!\!\sup_{(s,t) \in S(u)\times S(u)} h([u](s), [u](t)),
\] where, as in \eqref{Gdef}, $u$ is a piecewise constant function on $I$ and $S(u)$ denotes the set of discontinuity points of $u$.
\color{cyan} In fact, under natural assumptions on the supremand $h:(\mathbb R \setminus\{0\})\times (\mathbb R\setminus \{0\})\to \mathbb R$, which we refer to Definition \ref{defsymdiag} for,  we detect necessary and sufficient conditions on $h$ for the $L^1$-lower semicontinuity of $H$, namely, we prove the following result:
\begin{theorem}\label{asprop3.1ABC} Let $h:(\mathbb R\setminus \{0\}) \times (\mathbb R\setminus \{0\}) \to \mathbb R$ be diagonal and symmetric and let $H$ be the functional defined by \eqref{Hdef}.
$H$ is lower semicontinuous with respect to $L^1(I)$ convergence if and only if $h$ is lower semicontinuous and 
Cartesian submaximal.
\end{theorem}
\color{black}

\section{Notation and preliminaries}
For the reader's convenience, we mainly adopt the notation in \cite{ABC}.
Let $I \subset \R{}$ be open and bounded. We recall that a function $u \in L^1(I)$ is
a function of bounded variation if its distributional derivative $Du$ is a measure
in $I$. In this case $u$ is approximately differentiable almost everywhere and its
a.e. derivative is denoted by $u'$; moreover the set of approximate discontinuity
points (or jump set) of $u$, denoted by $S(u)$, is at most countable and at each point
$t \in S(u)$ there exist the right hand and left hand approximate limits, denoted by
$u(t^\pm)$. We denote by $[u](t)$ the difference between the two sides approximate limits, i.e. $[u](t):= u(t^+)- u(t^-)$,
the so called jump of $u$ at $t,$  as in \eqref{applim}. With this notation, the distributional derivative $Du$ admits the
decomposition
$Du = u'\lfloor {\mathcal L}^1 + D^su = u'\lfloor {\mathcal L}^1 +
\sum_{t \in S(u)}
(u(t^+) - u(t^-))\delta_t + D^cu$,
where $\mathcal L^1$ stands for the one-dimensional Lebesgue measure, $\delta_t$ is the Dirac delta
at $t$, and $D^cu$, the Cantor part of $Du$, is a non-atomic measure which is orthogonal
to the Lebesgue measure. The notation $D^su$ stands for the singular part of the
measure $Du$ (with respect to the Lebesgue measure). The space of functions of
bounded variation in $I$ is denoted by $BV (I).$ Its subset, constituted by functions which $D^c u$ is null, is denoted by $SBV(I)$ and the ones where $u'$ is also null is constituted by functions which take countably many constant values, i.e. such that $|D^j u|(I) < +\infty$.
Furthermore, when the set where $u$ takes constant values has finite perimeter, we say, as in \cite{BC}, that $u \in SBV_0(I)$, namely the singular set of $u$ is a Caccioppoli's partition, equivalently the set $S(u)$ has finite perimeter, i.e. $\mathcal H^0(S(u))<+\infty$.
In the subsequent analysis, we will focus on this class of functions, which, as proven in \cite{AFP} can be identified with the set of piecewise constant functions, also denoted by $PC(I)$.  In the following, as in \cite{ABC}, We will also adopt this notation.

For the sake of exposition, $I:=(a,b)$, $a,b \in \mathbb R$, $a<b$,
\begin{definition}\label{PCfunct_def}
A function $u : (a, b) \to \mathbb R$ is said to be piecewise constant on $(a, b)$ if there exist
points $a = t_0 < t_1 < \dots < t_N < t_{N+1} = b$ such that
\begin{equation}\label{3.11} u(t) \hbox{ is constant a.e. on }(t_{i-1}, t_i) \hbox{ for all }i = 1, 2, \dots ,N + 1.
\end{equation}
\end{definition}
Then $S(u)$ can be seen as the minimal set $\{t_1, t_2, \dots, t_N\}\subset (a, b)$ such that \eqref{3.11} holds. The subspace of $L^1(a, b)$ of all such $u$ is denoted by $PC(a, b)$.

We aim at studying the lower semicontinuity with respect to the $L^1$ topology of the following nonlocal functional
$H(u) : PC(I) \to \mathbb R$ of the form
\begin{equation}\label{Hdef}
H(u) = \;\;\;\;{\rm ess}\!\!\!\!\!\!\!\!\!\!\!\!\sup_{(t ,s) \in S(u)\times S(u)}h([u](t), [u](s)), 
\end{equation}
where $h : (\mathbb R \setminus\{0\})\times (\mathbb R\setminus \{0\}) \to \overline{\mathbb R}.$

We observe that there is no loss of generality in assuming $h$ bounded, since, as emphasized in \cite{ABP}, the analysis regarding the behaviour of the functional $H$ is not affected by replacing $h$ by the composition $f \circ h$, with $f$ strictly increasing, such as $f(t)= \arctan t$.

As mentioned earlier, the goal is to characterize the lower semicontinuity of the functional $H$.

 It is also possible to define $h$ at points of the type $(\zeta,0)$ or $(0,\theta)$, with an abuse of notation, imposing that $h(\zeta,0)=h(0,\theta):=\inf_{(\xi, \eta) \in (\mathbb R\setminus\{0\}) \times (\mathbb R \setminus \{0\})}h(\xi,\eta).$
Clearly, it is possible to consider $h$ defined in all $\mathbb R \times \mathbb R$, tacitly assuming the above values in $(\zeta, 0)$ and $(0,\theta)$, $\zeta, \theta \in \mathbb R$. 
With this choice, the functional $H$ in \eqref{Hdef} can be replaced by 
\begin{equation*}
H(u) = {\rm ess}\!\!\!\!\!\!\sup_{(t,s) \in I\times I}h([u](t), [u](s)), 
\end{equation*}
without affecting the analysis that we will present below. 
 \color{black}

\section{Lower semicontinuity} 
We start by showing some intermediate results that are important for the proof of  Theorem \ref{asprop3.1ABC} below. \color{black}

 As first observed in the integral settind in \cite{P1,P2} when dealing with nonlocal functionals, and later expanded in \cite{KZ, KRZ} in the supremal framework, some symmetry of the densities involved in the nonlocal models appear naturally. 

For every $m \in \mathbb N$, given a set $E \subset \mathbb R^m \times \mathbb R^m$, following \cite[(1.6)]{KZ}, we introduce  \begin{align*}\label{hat E}\hat E:=\{(\xi, \eta) \in E: (\xi,\xi), (\eta, \eta), (\eta,\xi) \in E\}.
	\end{align*} 

Let $f:\mathbb R \times \mathbb R\to \overline{\mathbb R}$, and $c \in \mathbb R$.

The $c$-(sub)level set of $f$ is defined as 
\begin{equation*}
	L_c(f):=\{(\xi,\eta)\in \mathbb R\times \mathbb R: f(\xi,\eta)\leq c\}.
	\end{equation*}

Given $h:(\mathbb R\setminus \{0\}) \times (\mathbb R\setminus \{0\}) \to \mathbb \overline{\mathbb R}$, define
\begin{equation}\label{hath} \hat{h}(\zeta, \eta) := \inf \{ c \in \mathbb{R}: (\zeta, \eta) \in \widehat{L_c(h)}, (\zeta, \eta) \in \mathbb{R} \times \mathbb{R} \}.
	\end{equation}
    \color{cyan}
\begin{definition}\label{defsymdiag}
A function $h:(\mathbb R\setminus \{0\}) \times (\mathbb R\setminus \{0\}) \to \overline{\mathbb R}$ is said to be {\it symmetric}
\color{cyan} if $h(\xi,\eta)= h(\eta,\xi)$ for every $\xi, \eta \in \mathbb R\setminus \{0\}$.

 A symmetric function $h:(\mathbb R\setminus \{0\}) \times (\mathbb R\setminus \{0\}) \to \overline{\mathbb R}$ is said to be {\it diagonal} if $h(\xi,\eta)= \hat{h}(\xi,\eta)$ for every $\xi, \eta \in \mathbb R\setminus \{0\}$.
\end{definition}
\color{black}
Observe that, as in \cite{ABC},  the functional $H$ in \eqref{Hdef} coincides with
\begin{align*}
 \mathcal H^0-{\rm ess}\!\!\!\!\!\!\!\!\!\!\!\!\!\sup_{(s,t) \in  S(u)\times S(u)} h([u](s), [u](t)) \in \overline{\mathbb R}.
	\end{align*}

We stress that there is no loss of generality in assuming $h$ diagonal and symmetric when evaluating $H$ in \eqref{Hdef}, and, as observed before, that $h$ is finite valued.

First, as proven in \cite[eq. (2.8)]{KRZ}, $\hat h$ in \eqref{hath} admits the following equivalent representation
\begin{align}\label{hhat}
\hat h(\xi, \eta)
=
\max\{ h(\xi,\xi), h(\eta,\eta), h(\eta,\xi), h(\xi,\eta)\}.
\end{align}

\begin{proposition}\label{propdiagsym}
Let $h:(\mathbb R\setminus \{0\}) \times (\mathbb R\setminus \{0\}) \to \mathbb R$ and let $\hat h$ be as in \eqref{hath}. Then
\begin{align*}
	\hat H(u):= \mathcal H^0-{\rm ess}\!\!\!\!\!\!\!\!\!\!\!\!\!\sup_{(t,s) \in S(u)\times S(u)} \hat h([u](s), [u](t))=H(u).
\end{align*}
\end{proposition}
\begin{proof}[Proof]
It can be easily verified that the symmetry assumption is not restrictive. It remains to check the diagonality. To this end, it immediately results that $h(\xi, \eta)\leq \hat h(\xi, \eta)$,
hence $H \leq \hat H$.
To prove the opposite inequality, we can observe that ${\mathcal H}^0- {\rm ess sup}$ in \eqref{Hdef} is indeed a maximum and is evaluated in a finite set of values, thus it will be a value of the type $h([u](\bar s), [u](\bar t))$ with $\bar s, \bar t \in S(u)$, but, clearly by the very definition of \eqref{Hdef}, it is greater than $h([u](\bar t), [u](\bar s)), h([u](\bar t), [u](\bar t)), h([u](\bar s), [u](\bar s))$, hence $h([u](\bar s), [u](\bar t)) = \hat h([u](\bar s), [u](\bar t))$,
 hence $\hat H \leq H$.
\end{proof}
 
 \bigskip
In view of the above result, in the sequel, when not explicitly mentioned, we will assume that $h$ is diagonal and symmetric.

\begin{definition}\label{Csubmaxdef}
Let $h:(\mathbb R\setminus \{0\}) \times (\mathbb R\setminus \{0\}) \to \mathbb R $ be symmetric and diagonal. We say that $h$ is {\it Cartesian submaximal} if 
\begin{equation}\label{new condition}
h(w_1 + w_2, y) \leq \text{max} \{ h(w_1, y), h(w_2, y), h(w_1, w_2)\},
\end{equation}
for every $w_1, w_2, y \in \mathbb R$.
\end{definition}

\begin{remark}\label{new rem}
It is also possible to consider Cartesian submaximal functions $h :\mathbb R^+\times \mathbb R^+ \to \mathbb R$. Then a Cartesian submaximal extension $\tilde h$ of such functions to $(\mathbb R\setminus\{0\})\times (\mathbb R \setminus \{0\})$ could be the following
\begin{equation*}
\tilde{h}(\xi,\eta):=\left\{ 
\begin{array}{ll}
h(\xi,\eta) & \hbox{ if } (\xi,\eta) \in \mathbb R^+ \times \mathbb R^+,
\\
{\rm ess }\sup_{(\xi,\eta) \in \mathbb R^+ \times \mathbb R^+} h(\xi,\eta) &\hbox{ otherwise,}
\end{array}
\right.
\end{equation*}
\end{remark}

Condition \ref{new condition} characterizes the $L^1(I)$-lower semicontinuity of $H$ in \eqref{Hdef}.
In fact, the following result holds.

\begin{proof}[Proof of Theorem \ref{asprop3.1ABC}]
	
	Let $H$ be $L^1(I)$-lower semicontinuous, then, taking two sequences $w_n \not = v_n \in \mathbb R$, converging to $w \not =v \in \mathbb R$ respectively and defining for every $t \in I$ and $n \in \mathbb N\cup\{+\infty\}$,
	\begin{equation*}
    u_n(t):=\left\{\begin{array}{ll}
		z &\hbox{ if } t \leq t_0,\\
		z + w_n & \hbox{ if } t_0 < t \leq t_1,\\
		z +w_n + v_n & \hbox{ if } t > t_1,
	\end{array}
	\right. \hbox{ and } u_\infty(t):=\left\{\begin{array}{ll}
		z &\hbox{ if } t \leq t_0,\\
		z + w & \hbox{ if } t_0 < t \leq t_1,\\
		z +w + v & \hbox{ if } t > t_1,
	\end{array}
	\right.
	\end{equation*}
	  with $t_0 < t_1 \in I$,
     and taking into account that $u_n \to u_\infty$ in $L^1(I)$,  $H(u_n)=  h(w_n, v_n)$,  and $H(u_\infty)= h(w,v)$, it results that 
	$h(w,v)\leq \liminf_n h(w_n, u_n)$,
	hence $h$ is lower semicontinuous.
	
	Moreover $h$ turns out to be Cartesian submaximal, indeed,  consider $y, w_1,w_2 \in \mathbb R$ such that we can consider the sequence 
	$$
	u_n(t):=\left\{\begin{array}{ll}
		z &\hbox{ if } t \leq t_0,\\
		z + y & \hbox{ if } t_0 < t \leq t_1,\\
		z + y+ w_1& \hbox{ if } t_1 \leq t \leq t_1+ \frac{1}{n},\\
		z+ y+ w_1+w_2  &\hbox{ if } t > t_1 + \frac{1}{n}.
	\end{array}
	\right.
	$$
	Clearly $u_n \to u_\infty$ in $L^1(I)$, where $$u_\infty(t)=\left\{\begin{array}{ll}
		z &\hbox{ if } t \leq t_0,\\
		z + y & \hbox{ if } t_0 < t \leq t_1,\\
		z + y +w_1+w_2 & \hbox{ if } t > t_1,
	\end{array}
	\right.$$
	Then, by the $L^1(I)$-lower semicontinuity of $H$,  
	\begin{align*}
	\supess_{s, t \in S(u)}h([u_\infty](s), [u_\infty](t))= & h(y, w_1+w_2)\leq \\
	\liminf_{n \to +\infty}\supess_{s, t \in S(u)}h([u_n](s), [u_n](t))= &\max\{h(y, w_1), h(y, w_2), h(w_1,w_2)\},
	\end{align*}
where in the above equalities we have exploited the symmetry and diagonality of $h$.

In order to prove the opposite implication, assume that $h$ is Cartesian submaximal and lower semicontinuous.

With the same construction as in \cite[Proof of Theorem 3.1]{ABC}, take $u_n$  converging to $u_\infty$ in $L^1_{loc}(I)$, hence, up to a not relabelled subsequence, also converging a.e. in $I$.
Since $u_\infty \in PC(I)$, there exists $\varepsilon > 0$ such that
$\varepsilon < \inf\{|t-s| : t, s \in S(u_\infty), t \not= s\}$. Fix $t, t' \in S(u_\infty)$, we can suppose that
$u_n(t \pm \varepsilon) \to u_\infty(t \pm \varepsilon) = u_\infty(t\pm)$,
$u_n(t' \pm \varepsilon) \to u_\infty(t' \pm \varepsilon) = u_\infty(t'\pm)$
and that, for all $n$, $(t \pm \varepsilon), (t' \pm \varepsilon)  \not \in  S(u_n).$

By the lower semicontinuity of $h$ it results 
$$
h([u_\infty](t), [u_\infty](t'))= h([u_\infty](t \pm \varepsilon), [u_\infty](t'\pm \varepsilon)) \leq \liminf_{n\to +\infty} h(u_n(t + \varepsilon)- u_n(t-\varepsilon), u_n(t'+\varepsilon)- u_n(t'-\varepsilon))
$$ 
By the Cartesian submaximality, the right hand side can be estimated as
\begin{align*}
 h(u_n(t + \varepsilon)- u_n(t-\varepsilon), u_n(t'+\varepsilon)- u_n(t'-\varepsilon))\leq \\
 \sup_{s, s'\in S_{u_n,\varepsilon, t, t'}}
 h([u_n](s), [u_n](s')) \leq 
 	\sup_{s,s' \in S(u_n)} h([u_n](s), [u_n](s')),
\end{align*}
\color{cyan} where the set $S_{u_n,\varepsilon, t, t'}:=(S(u_n)\cap (t-\varepsilon, t+\varepsilon) )\cup
 		(S(u_n)\cap (t'-\varepsilon, t' +\varepsilon))$. \color{black}
Consequently, taking the supremum as $t, t' \in S(u)$ on the left hand-side, we have the desired inequality.
$$
\sup_{t, t'\in S(u)}h([u](t), [u](t'))\leq \liminf_{n\to +\infty} \sup_{t, t' \in S(u_n)} h([u_n](t), [u_n](t')).
$$
This concludes the proof.
\end{proof} 

We conclude by making some comments and presenting some examples devoted to a better understanding of the notion of Cartesian submaximality.
\color{black}

\begin{remark}
\begin{itemize}
\item The composition of a submaximal function with a subadditve one is not Cartesian submaximal as the following example shows
$h(x,y):=\frac{|\sin (x+y)|}{x+y}$: indeed $h= g \circ +$, with $g:\mathbb R\setminus \{0\} \to \mathbb R$ defined as $g(x):=\frac{|\sin x|}{x}$ and $+:\mathbb R\times \mathbb R\to \mathbb R$;
\color{cyan}
It suffices to consider $x,y,w=\pi/2$ to verify $1=h(x+y,w)> \max\{h(x,y), h(y,w), h(x,w)\}= 0$.
\color{black}
\item	Each function $h: \mathbb R\times \mathbb R \to (0,+\infty)$ which is separately submaximal, i.e. such that $h(y, w_1+w_2)\leq \max \{h(y, w_1), h(y,w_2)\}$ is also Cartesian submaximal. 
	In particular, \color{cyan}if a symmetric function $h: \mathbb R^+ \times \mathbb R^+ \to (0,+\infty)$ is such that $h(\cdot, y)$ is decreasing for  a.e. $y$ and we extend it  
    as
    \begin{equation}\label{hstar} h^\star(w, y) = \left\{\begin{array}{ll}
h(w,y), &\hbox{ if } (w, y)\in \mathbb R^+\times\mathbb R^+,\\
0 &\hbox{ if } w =0\hbox{ or }y =0,\\
 +\infty &\hbox{ otherwise,}
\end{array}
\right.\end{equation}
     then $h^\star$ is separately submaximal.
     As an example one can consider $h(w,y)=\frac{1}{w+y}$ for $w,y \in \mathbb R^+$ and $h^\star$ as in \eqref{hstar}.
     \color{black}
	
	The same holds if $h$ is the ratio of a subadditive and superadditive function in each variable.

 \item  Examples of separately submaximal functions defined in $\mathbb R \times \mathbb R$ with values in $\overline{\mathbb R}$ (and hence Cartesian submaximal) resulting immediately from \cite[Example 3.3]{ABC}:
\[ h(w, y) = \left\{\begin{array}{ll}
\max\left\{ \frac{|\sin w|}{w}, \frac{| \sin y|}{y} \right\}, &\hbox{ if } (w,y)\in \mathbb R^+\times\mathbb R^+,\\
\color{cyan}
0 &\hbox{ if } w=0 \hbox{ or } y=0, \\
\color{cyan} +\infty &\hbox{ otherwise,}
\end{array}
\right.\]

\color{black}
and 
\[ h_\alpha(w,y) = \left\{\begin{array}{ll}  \alpha\frac{|\sin w|}{w} + (1 - \alpha) \frac{|\sin y|}{y} &\hbox{ if } (w,y)\in \mathbb R^+\times\mathbb R^+, \\
\color{cyan}
0 &\hbox{ if } w=0 \hbox{ of }y=0 ,\\
\color{cyan} +\infty &\hbox{ otherwise.}
 \color{black}
\end{array}
\right. \]
$\alpha \in (0,1).$

\end{itemize}

\end{remark}

\bigskip

\bigskip
{\noindent{\bf Acknowledgments.}

The research of J. M. was funded by FCT/Portugal through project UIDB/04459/2020 with DOI identifier 10-54499/UIDP/04459/2020.
E. Z. acknowledges the support of Piano Nazionale di Ripresa e Resilienza (PNRR) - Missione 4 ``Istruzione e Ricerca''
- Componente C2 Investimento 1.1, ``Fondo per il Programma Nazionale di Ricerca e
Progetti di Rilevante Interesse Nazionale (PRIN)" - CUP 853D23009360006.  She is a member of the Gruppo Nazionale per l'Analisi Matematica, la Probabilit\`a e le loro Applicazioni (GNAMPA) of the Istituto Nazionale di Alta Matematica ``F.~Severi'' (INdAM).
She also acknowledges the partial support of the INdAM - GNAMPA Project ``Metodi variazionali per problemi dipendenti da operatori frazionari
isotropi e anisotropi'' coordinated by Alessandro Carbotti, CUP ES324001950001.

\end{document}